\documentclass[oneside, dvipsnames, 11pt]{amsart}

\usepackage[T2A,T1]{fontenc}
\usepackage[utf8]{inputenc}
\usepackage[english]{babel}

\usepackage{xcolor}
\usepackage{hyperref}

\definecolor{myblue}{RGB}{60, 80, 197}
\hypersetup{%
  colorlinks=true,
  linkcolor=OrangeRed,
  citecolor = myblue,
  urlcolor = Blue
}

\usepackage{tabularx}
\usepackage{array} 
\usepackage{ragged2e}

\usepackage{mathtext}
\usepackage{amsmath,amsthm,mathtools,amssymb}

\theoremstyle{definition}
\newtheorem{definition}{Definition}[section]
\newtheorem{proposition}[definition]{Proposition}
\newtheorem{corollary}[definition]{Corollary}
\newtheorem{theorem}[definition]{Theorem}

\theoremstyle{remark}
\newtheorem*{remark}{Remark}

\makeatletter
\newtheorem*{rep@theorem}{\rep@title}
\newcommand{\newreptheorem}[2]{%
\newenvironment{rep#1}[1]{%
 \def\rep@title{{\bf #2 \ref{##1}}}%
 \begin{rep@theorem}}%
 {\end{rep@theorem}}}
\makeatother
\newreptheorem{theorem}{Theorem}

\usepackage[left=30mm, right=30mm, bottom=25mm, top=25mm,includefoot]{geometry}
\usepackage{adjustbox}
\usepackage{epigraph}
\usepackage{textcase}
\usepackage{version}
\excludeversion{confidential}

\usepackage{caption}
\usepackage[colorinlistoftodos]{todonotes}

\usepackage{tikz}
\usepackage{tikz-cd} 
\usetikzlibrary{cd}
\usepackage{pgfplots}
\pgfplotsset{width=6cm,compat=newest}

\usepackage{pdfpages}

\begin{document}

  \title[On Simplicial Maps with $H(f) \neq 0$]{On Simplicial Maps with Non-Zero Hopf Invariant}

  \author{Sergei Vad. Fomin}

  \date{}

    \thanks{This work was supported by the Ministry of Science and Higher Education of the Russian Federation (agreement 075-15-2025-344 dated 29/04/2025 for Saint Petersburg Leonhard Euler International Mathematical Institute at PDMI RAS)}

  \address{Sergei Vadimovich Fomin: St. Petersburg Department of Steklov Institute of Mathematics, nab. Fontanka, 27, St. Petersburg, 191023, Russia}

  \email{sergei.vadimovich.fomin@gmail.com}

  \maketitle

  \tableofcontents
    
  \section{Introduction}
    
   Let $f\colon K \rightarrow L$ be a simplicial map between triangulated manifolds, i.e. the geometric realizations $|K|$ and $|L|$ are \textit{homeomorphic} to topological manifolds. Generally, it is interesting to establish various bounds on the complexity of the map $f$ under various topological constraints.
    
    A classic paper of S. Eilenberg provides a criterion for the vanishing of the Hopf invariant of a map between simplicial spheres in terms of the dimension of a preimage~\cite[Theorem II]{eilenberg_continuous_1940}. This result was used by K. V. Madahar and K. S. Sarkaria~\cite{madahar_minimal_2000},~\cite{madahar_simplicial_2002} to show that the triangulation of the Hopf fibration they constructed is minimal in terms of the number of vertices.

    The goal of this paper is to establish an analogous bound on simplicial complexity of the preimage of a \textit{regular point}\footnote{That is, a point lying in the interior of a maximal simplex.} for any simplical map (say, between triangulated spheres) with non-zero Hopf invariant. 
    
    \begin{definition}
        Let $K, L$ be simplicial complexes, and let $s$ be a fixed maximal simplex of $L$. Then
        \[N(f,s) := \#\{\sigma\in \operatorname{Max}(K)\,|\, f(\sigma)=s\},\]
        where $\operatorname{Max}(K)$ is the set of maximal simplices of the complex $K$.
    \end{definition}

    The main result is the estimate $N(f,s) > \frac{n^3}{6}$ for a simplicial map $f\colon K \rightarrow L$ from a triangulated $(2n-1)$-dimensional sphere to a triangulated $n$-dimensional sphere with non-zero Hopf invariant (Theorem~\ref{lower}).

    We also construct a family of examples demonstrating that $N(f,s)$ does not necessarily grow without bound as the Hopf invariant increases (Theorem~\ref{upper}). To construct the maps, we triangulate the Pontrjagin construction and use information about the image of $J$-homomorphism.
    
    In a previous joint paper~\cite{bludov_essential_2025} with M. V. Bludov and O. R. Musin, we considered the basic case of a simplicial map from a triangulated three-dimensional sphere to a triangulated two-dimensional sphere.\footnote{There, this value $N(f,s)$ was denoted as $\mu(f,s)$, as in a previous paper by Musin~\cite{musin_homotopy_2025}. The papers~\cite{bludov_essential_2025} and~\cite{musin_homotopy_2025} have definitions for $\mu(f,s)$ which coincide in the situation of codimension one, as considered in those papers, but differ in case of codimension larger that one, so we had to change notation.} In that case, we obtained the estimate $N(f,s) \geq 9$ for essential maps and, as
	a consequence, we proved that the triangulation of Hopf fibration constructed by K. V. Madahar and K. S. Sarkaria is not only minimal in terms of the number of vertices but also in terms of the number of tetrahedra.\footnote{As is clear from a double-counting argument and the Euler's formula, the triangulation is also minimal in triangles and in edges.} Furthermore, a family of examples was constructed, demonstrating that the estimates cannot be improved under stronger restrictions on the values of the Hopf invariant. 

\section{Preliminaries}

\subsection{Rational manifolds}
    A convenient level of generality for the formulation and proof of the main result is rational (homological) manifolds, so we recall some basic definitions.
    
    The are variants of the definition of a rational manifold which assume different topological regularity properties, such as finite cohomological dimension~\cite{borel_seminar_1960}. Following the definition in~\cite{milnor_characteristic_1974}, we assume our space to be the geometric realization of a simplicial complex. Moreover, all rational homology manifolds under consideration are assumed to be closed. Therefore, the complexes are necessarily finite.
 
\begin{definition}
    A finite simplicial complex $K$ is called a rational homology $n$-manifold if, for every point $x \in |K|$, the rational local homology groups $H_i(|K|, |K| \setminus \{x\})$ are zero for $i \neq n$ and isomorphic to $\mathbb{Q}$ for $i = n$.
\end{definition}

    Following~\cite[§20]{milnor_characteristic_1974}, we recall some basic properties of pseudomanifolds and (triangulated) rational homology manifolds.
\begin{definition}
    A finite simplicial complex $K$ is called a pseudomanifold\footnote{In~\cite{milnor_characteristic_1974}, pseudomanifolds are referred to as \textit{$n$-circuits}.} of dimension $n$ (or an $n$-pseudomanifold) if
    \begin{enumerate}
        \item every simplex is a face of an $n$-dimensional simplex;
        \item every $(n-1)$-dimensional simplex is a face of exactly two $n$-dimensional simplices;
        \item (\textit{strong connectedness}) for any two $n$-simplices $\sigma$ and $\tau$, there exists a sequence
        \[\sigma=\sigma_1, \sigma_2, \dots, \sigma_s=\tau,\]
        where $\sigma_i$ and $\sigma_{i+1}$ share a common $(n-1)$-dimensional face.
    \end{enumerate}
\end{definition}

\begin{definition}
    An orientation of an $n$-pseudomanifold is a choice of orientations for its maximal $n$-simplices such that for any $(n-1)$-simplex the two incident $n$-simplices induce opposite orientations. An $n$-pseudomanifold is orientable if it admits an orientation.

    An oriented $n$-pseudomanifold is an $n$-pseudomanifold with a fixed orientation.
\end{definition}

\begin{remark}
    The sum of $n$-simplices, oriented as in the previous definition, represents a homology class $[M]$ in $H_n(M; \mathbb{Z})$, which is called the fundamental class of an oriented (closed) $n$-pseudomanifold.
\end{remark}

\begin{proposition}
    Let $M$ be a (closed) orientable $n$-pseudomanifold. Then $H_n(X)$ is isomorphic to $\mathbb{Z}$, and any fundamental class (corresponding to any orientation) is a generator. Furthermore, $H^n(X)\cong\operatorname{Hom}(H_n(X), \mathbb{Z})\cong \mathbb{Z}$. 
\end{proposition} 
\begin{proof}
    The homology statement is proved in~\cite[Chapter 17.2]{fomenko_homotopical_2016}, and also in~\cite[Chapter 24]{seifert_1980}. Moreover, it is proven in Theorem II of~\cite[Chapter 24]{seifert_1980} that $H_{n-1}(X)$ is free abelian. Then, the cohomology statement follows from the universal coefficient theorem.
\end{proof}

\begin{corollary}
     Let $M$ be a (closed) orientable $n$-pseudomanifold. Then $H_n(X; \mathbb{Q})\cong \mathbb{Q}$, $H^n(X;\mathbb{Q})\cong\operatorname{Hom}(H_n(X;\mathbb{Q}), \mathbb{Q})\cong \mathbb{Q}$.
\end{corollary}

\begin{proposition}
    Every connected component of a rational homology $n$-manifold is an $n$-pseudomanifold.\footnote{For a proof, one can check that the proof of a corresponding statement for integral pseudomanifolds in~\cite[Chapter 17.3]{fomenko_homotopical_2016} works for the rational homology manifolds as well.}
\end{proposition}

For oriented rational homology manifolds, Poincaré duality holds with rational coefficients~\cite{borel_seminar_1960},~\cite[Theorems~V.9.2,~V.9.3, Corollary~V.10.2]{bredon_sheaf_1997}.

\begin{proposition}\label{poincare}
    Let $M$ be an oriented rational homology manifold, and let $N \subset M$ be a polyhedron. Then the map 
    \[D \colon H^k(M | N; \mathbb{Q}) \rightarrow H_{n-k}(N; \mathbb{Q}), \quad \alpha \mapsto [M] \cap \alpha,\]
    is an isomorphism. Setting $N=M$, we get that the map
    \[D \colon H^k(M; \mathbb{Q}) \rightarrow H_{n-k}(M; \mathbb{Q}), \quad \alpha \mapsto [M] \cap \alpha,\]
    is also an isomorphism.
\end{proposition} 

\begin{remark}
    It is easy to see that the isomorphism $D$ is natural with respect to inclusions of subsets, i.e. for any polyhedra $N_1 \subset N_2$ in $M$, the following diagram commutes:
    \[ \begin{tikzcd}
        H^k(M|N_1;\mathbb{Q}) \arrow{r}{D} \arrow{d} & H_{n-k}(N_1; \mathbb{Q}) \arrow{d} \\
             H^k(M|N_2;\mathbb{Q}) \arrow{r}{D}&  H_{n-k}(N_2; \mathbb{Q}) 
            \end{tikzcd}
    \]
\end{remark}

    Finally, we note that the class of rational homology manifolds behaves nicely under simplicial maps.
    \begin{proposition}\label{preimage} 
        Let $K$ and $L$ be rational homology manifolds of dimensions $n$ and $m$, respectively, let $f \colon K \rightarrow L$ be a simplicial map, and let $p$ be a point in the interior of a maximal simplex of the complex $L$. Then $f^{-1}(p)$ is a rational homology manifold of dimension $n-m$.\footnote{In particular, it admits a triangulation.} If $K$ and $L$ are oriented, then $f^{-1}(p)$ admits an orientation that is compatible\footnote{That is, for any $q \in f^{-1}(p)$, the Künneth map on local homology groups $H_{n-m}(f^{-1}(p) \mid q; \mathbb{Q}) \times H_m(L \mid p; \mathbb{Q}) \rightarrow H_n(K \mid q; \mathbb{Q})$ maps the generators $([f^{-1}(p)]|_q, [L]|_p)$ to the generator $[K]|_q$.} with the orientations of $K$ and $L$.\footnote{This proposition is a step in the proof of Lemma 20.3 in~\cite{milnor_characteristic_1974}.}
    \end{proposition}

\subsection{Functional cup product}
The functional cup product was introduced by Steenrod in~\cite{steenrod_cohomology_1949} as a generalization of the Hopf invariant. It can be used to define the Hopf invariant, for instance, for maps between (oriented) rational homology manifolds.

\begin{definition}
    Let $A$ be a ring, and let $f \colon X \rightarrow Y$ be a continuous map. Suppose the classes $[u] \in H^k(Y;A)$ and $[v] \in H^\ell(Y;A)$ are such that $f^\ast [u] = 0$ and $[u] \cup [v] = 0$. Let $a \in C^{k-1}(X;A)$ and $b \in C^{k+\ell-1}(Y;A)$ be cochains such that $\delta a = f^\ast u$ and $\delta b = u \cup v$. Then the functional cup product is defined as
    \[u \cup_f v := a \cup f^\ast v - f^\ast b.\]
    At the cohomology level, the functional cup product is well-defined\footnote{That is, independent of the choice of the cochains $a$ and $b$, and of the representing cochains of the classes $[u]$ and $[v]$.} as a map into the quotient group
    $H^{k+\ell-1}(X;A) / \big(f^\ast H^{k+\ell-1}(Y;A) + H^{k-1}(X;A) \cup f^\ast [v]\big)$.

\end{definition}
    Let $A = \mathbb{Z}$ or $\mathbb{Q}$, let $p = q = n > 1$, and let the spaces $X = |K|$ and $Y = |L|$ be oriented rational homology manifolds of dimensions $2n-1$ and $n$, respectively, where $L$ is connected and non-empty, and the map $f \colon X \rightarrow Y$ is such that $f^{\ast}|_{H^n(Y; A)} = 0$. Let $e \in H^n(Y;A)$ be the class on $Y$ defined by the relation $\langle e, [Y]\rangle = 1$. Then, in the definition of the functional product, taking the quotient is not required: $e \cup_f e \in H^{2n-1}(X; A)$.
    
    The number $H(f) \in A$ defined by
    \[ H(f) := \langle e \cup_f e, [X]\rangle \]
    will be called the \textbf{Hopf invariant} of the map $f$ between oriented rational homology manifolds of dimensions $2n-1$ and $n$ satisfying $f^{\ast}|_{H^n} = 0$.

    \begin{remark}
        It is straightforward to verify that the definitions of the Hopf invariant for $A=\mathbb{Z}$ and $A=\mathbb{Q}$ are compatible. In other words, for any map $f$ satisfying $f^\ast|_{H^n_\mathbb{Q}} = f^\ast|_{H^n_\mathbb{Z}} \otimes \mathbb{Q} = 0$, the Hopf invariant $H(f) \in \mathbb{Q}$ is well-defined; moreover, if it turns out that $f^\ast|_{H^n_\mathbb{Z}} = 0$, then $H(f)$ is an integer.
    \end{remark}
    \begin{remark}
        The functional cup product is homotopy invariant~\cite[Corollary 10.2]{steenrod_cohomology_1949}.
    \end{remark}
    \begin{remark} The functional cup product is graded-commutative~\cite[Theorem 6.2]{steenrod_cohomology_1949}. That is, for $u \in H^p(Y;\mathbb{Q})$, $v \in H^q(Y;\mathbb{Q})$, and $f \colon X \rightarrow Y$ such that $u \cup v = 0$ and $f^\ast u = f^\ast v = 0$, we have:
    \[u \cup_f v = (-1)^{pq} \, v \cup_f u.\]
    \end{remark}
    It follows from the previous remark that the Hopf invariant is identically zero for odd $n$, meaning that the main estimate is interesting only when $n$ is even.

\subsection{\texorpdfstring{$\operatorname{Im} J$}{Im J} and Hopf invariant}
    \begin{definition}
        Let $n > 0$ and $k \in \mathbb{Z}$. The number $k$ is called admissible if $k \in \operatorname{Im} H$, where
        \[H \colon \pi_{2n-1}(S^n) \rightarrow \mathbb{Z}\]
        is the Hopf invariant.
    \end{definition}
    \begin{remark}
        It is widely known~\cite[Section 4.B]{hatcher_algebraic_2002} that the sets of admissible values for various $n$ are as follows:
        \begin{itemize}
            \item if $n$ is odd, only the value 0 is admissible;
            \item if $n = 2, 4, 8$, all integers are admissible;
            \item for all other $n$, all even values are admissible.
        \end{itemize}
    \end{remark}

    If $S^{k} \subset S^{N}$ ($N>k$) is a standardly embedded sphere, its normal bundle is canonically trivial. Therefore, any map $S^k \rightarrow \operatorname{SL}_n(\mathbb{R})$ determines a framing, and using the Pontrjagin construction on the framed submanifold $S^k$ yields an element of $\pi_{n+k}(S^n)$. One can verify (see, e.g., \cite{putman_homotopy_nodate}) that this assignment defines a homomorphism $J \colon \pi_k(\operatorname{SL}_n(\mathbb{R})) \rightarrow \pi_{n+k}(S^n)$, which is called the \textbf{$J$-homomorphism}.

    \begin{proposition}\label{j-real}
        Let $k \in \mathbb{Z}$ be an admissible number (for a fixed $n$). Then there exists an element in $\operatorname{Im} J$ with this Hopf invariant. In other words, $\operatorname{Im} H = \operatorname{Im} (H \circ J)$.
    \end{proposition}
    \begin{proof}
        For $n=2,4,8$, the map with Hopf invariant one can be chosen to be a smooth fiber bundle whose fibers are standardly embedded spheres~\cite[Examples 4.45--4.47]{hatcher_algebraic_2002}. Under the Pontrjagin--Thom isomorphism, the class of this map corresponds to a standardly embedded sphere equipped with a framing. Therefore, for $n=2,4,8$, the image $\operatorname{Im} J$ contains a class with Hopf invariant one.

        It is also known that for even $n$, there exists a homotopy class in $\operatorname{Im} J$ with Hopf invariant two~\cite{wall},~\cite{whitehead}.
    \end{proof}    
\section{The main bound}
    In this section, all (co)chain complexes and (co)homology groups are assumed to have rational coefficients.
    
    Let $K$ be an oriented rational homology manifold of dimension $n$, let $L$ be a simplicial complex, let $s$ be a maximal simplex of $L$, and let $f\colon K\rightarrow L$ be a simplicial map.

    Let $Q$ be the set of all simplices in the complex $K$ that map surjectively onto the simplex $s$ under $f$. Since $s$ is a maximal simplex, the following property holds: if a face $\tau$ of a simplex $\sigma$ lies in $Q$, then $\sigma$ itself also lies in $Q$. By the geometric realization $|Q|$ of the set $Q$, we mean the union of the relative interiors of the (geometric realizations of) simplices in $Q$. It follows from the property described above that the set $|Q|$ is open.

    We call two simplices in $Q$ incident if one is a face of the other. The incidence relation generates an equivalence relation on $Q$, whose equivalence classes we denote by $Q_i$. These sets satisfy the same property as $Q$, and their geometric realizations $|Q_i|$ are the connected components of $|Q|$. 

    For any set of simplices $S$, we denote by $\bar{S}$ the induced subcomplex, i.e. the closure of $S$ under taking faces of simplices.

    Having fixed this notation, we are now ready to state the following theorem.
    
    \begin{proposition}
        Let $K$ and $L$ be oriented rational homology manifolds of dimensions $2n-1$ and $n$, respectively, where $L$ is connected and non-empty. Let $s$ be a maximal simplex of $L$, and let $f \colon K \rightarrow L$ be a simplicial map such that $f^\ast|_{H^n_\mathbb{Q}} = 0$. Suppose that for each $Q_i$ (defined as above) there exists a vertex $v_i$ of the simplex $s$ such that the set
        \[R_i := f^{-1}(v_i) \cap |\overline{Q_i}|\]
        has dimension less than $n-1$. Then $H(f) = 0$.
    \end{proposition}
    
    \begin{proof}
        Let $p \in \operatorname{Int}(|s|)$, and let $q \in |L|$ be any point outside $|s|$. Define $S_p = |L| \setminus \{p\}$, $S_q = |L| \setminus \{q\}$, $S'_p = |K| \setminus f^{-1}(p)$, and $S'_q = |K| \setminus f^{-1}(q)$. 
        Since the map \mbox{$H_0(p) \rightarrow H_0(|L|)$} is surjective, Theorem~\ref{poincare} implies that the map $H^n(|L| \mid p) \rightarrow H^n(|L|)$ is also surjective. Let us choose a cochain $u \in Z^n(|L| \mid p)$ such that the class $[u]$ equals $e$ in $H^n(|L|)$. Similarly, we choose $v \in Z^n(|L| \mid q)$ such that $[v]$ equals $e$ in $H^n(|L|)$. Then the cup product $u \cup v$ is defined as an element of the group $Z^{2n}(|L|, S_p + S_q)$. Since the map $C^\bullet(|L|, S_p \cup S_q) \rightarrow C^\bullet(|L|, S_p + S_q)$ is a quasi-isomorphism~\cite[Section 3.1, p. 204]{hatcher_algebraic_2002}, and $H^\bullet(|L|, S_p \cup S_q) = H^\bullet(|L|, |L|) = 0$, there exists a cochain $b \in C^{2n-1}(|L|, S_p + S_q)$ such that $\delta b = u \cup v$.
        
        Consider the following commutative diagram:

        \begin{center}      
        \begin{adjustbox}{width=\textwidth}
        \begin{tikzcd}
             C^n(|L| \mid p) \arrow{r}{f^\ast} \arrow{d} & C^n(|K| \mid f^{-1}(p)) \arrow{r}{\operatorname{in}^\ast} \arrow{d} & C^n(|K| \mid |\overline{Q}|) \arrow{d} & C^{n-1}(|K| \mid |\overline{Q}|) \arrow[swap]{l}{\delta} \arrow{d}  \arrow{r} & C^{2n-1}(|K| , (|K|\setminus |\overline{Q}|) + S'_q) \arrow{d}\\
             C^{n}(|L|)\arrow{r}{f^\ast} & C^n(|K|) \arrow[equal]{r} & C^n(|K|) & C^{n-1}(|K|) \arrow[swap]{l}{\delta} \arrow{r} & C^{2n-1}(|K|) \\[-20pt]
             u \arrow[mapsto]{r} & f^\ast u \arrow[equal]{r} & f^\ast u & a \arrow[mapsto]{l} \arrow[mapsto]{r} & u \cup_f v
            \end{tikzcd}
        \end{adjustbox}
        \end{center}

        In this diagram, the rightmost horizontal arrows in both rows represent the map \mbox{$(-)\cup f^\ast v-f^\ast b$}. If we trace the cocycle\footnote{More precisely, its image in $C^n(|L|)$.} $u$ through the bottom row of the diagram, we obtain the cocycle $u \cup_f v$, which represents the class $e \cup_f e$. If we can find a preimage $a'\in C^{n-1}(|K| \mid |\overline{Q}|)$ of the element $(f\circ\operatorname{in})^\ast u \in C^n(|K| \mid |\overline{Q}|)$ under $\delta$, then, as the cochain $a\in C^{n-1}(|K|)$ satisfying $\delta a = f^\ast u$, we can take the image of $a'$ in $C^{n-1}(|K|)$. In this case, it follows that $e \cup_f e$ is represented by a cocycle in $C^{2n-1}\big(|K| , (|K|\setminus |\overline{Q}|) + S'_q\big)$. Since the map $C^\bullet\big(|K|, (|K|\setminus |\overline{Q}|) \cup S'_q\big)\rightarrow C^\bullet\big(|K|, (|K|\setminus |\overline{Q}|) + S'_q\big)$ is a quasi-isomorphism, and $H^\bullet\big(|K|, (|K|\setminus |\overline{Q}|) \cup S'_q\big)=H^\bullet(|K|, |K|)=0$, we conclude that $e \cup_f e = 0$, and thus $H(f)=0$.

        Thus, it is enough to show that the map $\operatorname{in}^\ast\colon H^n(|K| \mid f^{-1}(p)) \rightarrow H^n(|K| \mid |\overline{Q}|)$ induced by the inclusion $f^{-1}(p) \subset |\overline{Q}|$ is zero. By Theorem~\ref{poincare}, it is enough to verify that the map $\operatorname{in}_\ast\colon H_{n-1}(f^{-1}(p)) \rightarrow H_{n-1}(|\overline{Q}|)$ is zero. It is clear that the sets $Z_i := f^{-1}(p) \cap Q_i$ are precisely the connected components of $f^{-1}(p)$. Therefore, it suffices to check that for each $i$, the map $\operatorname{in}_\ast\colon H_{n-1}(Z_i) \rightarrow H_{n-1}(|\overline{Q_i}|)$ is zero.
        
        Let $W_i := R_i \cup |Q_i|$. It is clear that $W_i$ deformation retracts onto $R_i$, which has dimension less than $n-1$. Thus, $H_{n-1}(W_i) \cong H_{n-1}(R_i) = 0$. On the other hand, the map $\operatorname{in}_\ast\colon H_{n-1}(Z_i) \rightarrow H_{n-1}(|\overline{Q_i}|)$ factors through $H_{n-1}(W_i)$. Therefore, it is zero, which completes the proof.
    \end{proof}
    \begin{remark}
        The proof of the theorem presented above is based on a geometric proof\footnote{In~\cite{bludov_essential_2025}, the proof is also geometric.} that has been translated into a cohomological proof "via Poincaré duality". Let the points $p$ and $q$ lie in the interiors of distinct maximal simplices of the complex $|L|$. As noted in Theorem~\ref{preimage}, the preimages $f^{-1}(p)$ and $f^{-1}(q)$ define $(n-1)$-cycles. The Hopf invariant can be computed via its classical definition as the linking number\footnote{The linking number can be defined for simplicial cycles; see~\cite{seifert_1980}.} of the preimages $\operatorname{lk}\big(f^{-1}(p), f^{-1}(q)\big)$~\cite{eilenberg_continuous_1940}. Under the assumptions of the theorem, one can then find a chain $C$ contained in $|\overline{Q}|$ (and therefore clearly disjoint from $f^{-1}(q)$) whose boundary is $\partial C = f^{-1}(p)$. More specifically, one can choose a chain supported on $\bigcup_i \overline{f^{-1}([p, v_i)) \cap |Q_i|}$. The conditions of the theorem guarantee that its boundary is indeed $f^{-1}(p)$. Since $C$ does not intersect $f^{-1}(q)$, it follows by definition that the Hopf invariant is zero. The chain $C$ corresponds to the cochain $a'$ in the cohomological proof above.
    \end{remark}

    \begin{remark}
        The previous theorem provides a refinement to Theorem II of~\cite{eilenberg_continuous_1940}, as it shows that one can replace $(ii)$ in Theorem II by a weaker condition.
    \end{remark}
    
    From this theorem, we obtain the bound.
    \begin{theorem} \label{lower}
        Let $K$ and $L$ be oriented rational homology manifolds of dimensions $2n-1$ and $n$, respectively, and let $s$ be a maximal simplex of the complex $L$. If $H(f) \neq 0$, then 
        \[N(f,s) > \frac{n^3}{6}.\]
    \end{theorem}
    \begin{proof}
        We may assume that $n$ is even and $n > 2$.
        
        Since $H(f) \neq 0$, there exists a component $Q_i$ such that for every vertex $v_k$ of the simplex $s = [v_0, \dots, v_n]$, the subcomplex $f^{-1}(v_k) \cap |\overline{Q_i}|$ has dimension at least $n-1$. This means that for each vertex $v_k$, there exists a $(2n-1)$-simplex $\sigma \in Q_i$ with $n$ of its vertices mapping to $v_k$. We call such a simplex $k$-extreme.
        
        Let us define a graph $G$ on the set $\operatorname{Max}(Q_i)$ of maximal simplices of $Q_i$: an edge is drawn between two vertices if the corresponding simplices share a common face of codimension 1. Since $Z_i$ is a pseudomanifold (and, in particular, is strongly connected), the graph $G$ is connected.

        For each maximal simplex $\sigma\in Q_i$, we define $\operatorname{vec}(\sigma)\in \mathbb{N}^{n+1}$ as the vector of counts of vertices of $\sigma$ in the preimages of the vertices of the simplex $s=[v_0, \dots, v_n]$:
        \[(\operatorname{vec}(\sigma))_k := \# \big(f^{-1}(v_k)\cap \sigma\big).\]

        Consider $\sigma=[u_0,\dots,u_{2n-1}]\in\operatorname{Max}(Q_i)$ and $\tau\in\operatorname{Max}(K)$ such that $\sigma\cap\tau=[u_0,\dots,\hat{u}_k, \dots, u_{2n-1}]$ with $(\operatorname{vec}(\sigma))_k\neq 1$. It is clear that in this case $\tau\in\operatorname{Max}(Q_i)$. Thus, the degree of any vertex in the graph $G$ is at least $n$. Furthermore, if a vertex does not correspond to a $k$-extreme simplex (for any $k$), its degree is at least $n+1$.

        Let $\Delta_{\operatorname{st}}=\{\mathbf{x}=(x_0,\dots, x_{n})\in \mathbb{R}^{n+1} \mid \sum_k x_k=2n,\, x_k\geq 1\}$, which is the simplex with vertices $\mathbf{y}_k:=(1,\dots, 1)+(n-1)\overrightarrow{e_k}$. Let $\mathcal{A}^k_j:=\{\mathbf{x}\in \Delta_{\operatorname{st}}\cap \mathbb{Z}^{n+1} \mid \|\mathbf{x}-\mathbf{y}_k\|_{\ell^1}=2j\}$. It is easy to see that for $0\leq k \leq n$ and $0\leq j \leq \frac{n}{2}-1$, the sets $\mathcal{A}^k_{j}$ are disjoint. 
        
        To each edge $e$ between vertices $\sigma$ and $\tau$ of the graph $G$, we associate the segment $\operatorname{vec}(e):=[\operatorname{vec}(\sigma), \operatorname{vec}(\tau)] \subset \mathbb{R}^{n+1}$. This yields a connected metric graph $G'\subset \Delta_{\operatorname{st}}$, where $G'=\operatorname{vec}(G)$. Note that for a $k$-extreme simplex $\sigma$, we have $\operatorname{vec}(\sigma)=\mathbf{y}_k$. If two simplices $\sigma_1,\, \sigma_2 \in \operatorname{Max}(Q_i)$ share a common face of codimension 1, then $\operatorname{vec}(\sigma_1)$ and $\operatorname{vec}(\sigma_2)$ either coincide or differ by one in exactly two coordinates. Consequently, all edges of the graph $G'$ have length 2 in the $\ell^1$-metric, and moving along an edge changes the distance to any $\mathbf{w}_k$ by either $0$ or $\pm 2$. By the connectedness of $G'$, it follows that $\mathcal{A}^k_{j} \cap \operatorname{V}(G') \neq \varnothing$ for all $0\leq k\leq n$ and $0\leq j < \frac{n}{2}$.

        For each pair $(k,j)$ with $0\leq k \leq n$ and $0\leq j \leq \left\lfloor\frac{n}{6} - \frac{2}{3}\right\rfloor$, we fix a vertex $w^k_j \in \operatorname{V}(G')\cap\mathcal{A}^k_{3j}$. Since the vertices $\{w^k_j\}_{k,j}$ are at a distance of at least 6 from one another, they do not share any common neighbors in $G'$. Letting $w^k_j=\operatorname{vec}(\sigma^k_j)$ for $\sigma^k_j\in\operatorname{V}(G)=\operatorname{Max}(Q_i)$, it follows that the vertices $\sigma^k_j\in\operatorname{V}(G)$ also share no common neighbors in $G$. Thus, in the graph $G$, we have found $(n+1)\left\lfloor\frac{n}{6}-\frac{2}{3}\right\rfloor$ vertices of degree at least $n+1$ (corresponding to the vertices $w^k_j$ with $j\neq 0$) and $n+1$ vertices of degree at least $n$ (corresponding to the vertices $w^k_0$), where all these vertices and their neighbors are distinct. Consequently,
        \[N(f,s)\geq \#\operatorname{Max}(Q_i)=\#\operatorname{V}(G)\geq f(n),\]
        where
        \[f(n)=(n+2)(n+1)\left\lfloor\frac{n}{6}-\frac{2}{3}\right\rfloor + (n+1)^2.\]
        It is straightforward to verify that $f(n)>\frac{n^3}{6}$ for all even $n>2$ except $n=8$. For the case $n=8$, we note that $\mathcal{A}^0_{\frac{n}{2}-1} \cap \operatorname{V}(G') \neq \varnothing$ by the connectedness of $G'$. Furthermore, it is easy to see that any $w\in \mathcal{A}^0_{\frac{n}{2}-1} \cap \operatorname{V}(G')$ is at a distance of at least $6$ from all the vertices $w^k_j$. Therefore, the vertex $\sigma$ of the graph $G$ such that $\operatorname{vec}(\sigma)=w$, along with all its neighbors in $G$ (of which there are at least $n+1=9$), have not been counted previously. One can readily check that $f(8)+10>\frac{8^3}{6}$.
    \end{proof}

    \begin{remark}
        This bound is clearly non-trivial in that without requiring non-triviality of the homotopy class $[f]$ one gets $N(f,s)=0$ for a constant map. Moreover, one could be tempted to try to prove a lower bound by some simple combinatorial argument which would only require that $N(f,s)\neq 0$ instead of a homotopical condition. This does not work as it is easy to come up with many examples where $N(f,s)$ is linear in $n$.

        For example, consider the $(n-1)$-th simplicial suspension $\Sigma^{n-1}g$ of a simplicial map $g:\partial \Delta^{n+1} \to \partial \Delta^{2}$ which sends one of the vertices of $\partial \Delta^{n+1}$ to some vertex $A$ of $\partial \Delta^{2}$ and the rest of the vertices of $\partial \Delta^{n+1}$ to some other vertex $B$ of $\partial \Delta^{2}$. Then for some top simplex $s$ we have $N(\Sigma^{n-1}g,s)=n+1$.

        It is not at all clear what happens with $N(f,s)$ when $[f]$ is non-trivial but $H(f)=0$.
    \end{remark}
    
\section{Examples}
    In this section, we describe a family of simplicial maps $f \colon S^{2n-1} \rightarrow S^n$ that realize all admissible values of the Hopf invariant and have bounded $N(f,s)$.

    Given a (smooth) framed $(m-n)$-submanifold $P^{m-n} \subset M^m$, the Pontrjagin construction returns a corresponding map $f \colon M^m \rightarrow S^n$:

    \[ \begin{tikzcd}
    P \times D^{n} \arrow[hook]{r}{\xi} \arrow[swap]{d}{\pi} & M \arrow{d}{f} \arrow[hookleftarrow]{r} & W \arrow{d} \\%
    D^{n} \arrow{r}{\phi}& S^{n} & \ast \arrow{l}
    \end{tikzcd}
    \]

    In this diagram, $\xi$ is an embedding such that $\xi|_{P \times \{0\}}$ is the inclusion $P \subset M$, $\pi$ is the projection onto the second factor, $\phi$ collapses the boundary of the ball, and $W$ is the closure of $M \setminus \xi(P \times D^{n})$.

    We now describe a way to realize the Pontrjagin construction simplicially.

    We define a simplicial map\footnote{There is a standard triangulation on the product $\Delta^{n} \times I$.} $\Phi \colon \Delta^{n} \times I \rightarrow \Delta^{n+1}$ as follows:
    \[\Phi(e_i, 0) = e_i, \quad \Phi(e_i, 1) = e_{n+1},\] 
    where $\Delta^k = [e_0, \dots, e_k]$ and $I = [0,1]$.

    Let $U = \Delta^{n} \times \{0\} \cup \partial\Delta^{n} \times I$; then $\operatorname{Im}\Phi|_U \subset \partial \Delta^{n+1}$. We triangulate $D^{n}$ and $S^{n}$ (smoothly) so that $D^{n} \cong |U|$, $S^{n} \cong |\partial \Delta^{n+1}|$. Let $\phi_\Delta = |(\Phi|_U)|$.

    Smoothly triangulate the manifold $P$. A choice of the orderings on the vertex set of the triangulation of $P$ and on the simplicial complex $U$ induces a smooth triangulation on $P \times D^{n}$ such that $\pi$ is simplicial. In particular, the boundary $\xi(P \times \partial D^{n})$ of the manifold $W$ is smoothly triangulated, $\phi_\Delta \circ \pi|_{P \times \partial D^{n}} = \operatorname{const}_{e_{n+1}}$. This triangulation on the boundary can be extended to a (smooth) triangulation of $W$~\cite[Theorem 10.6]{munkres_elementary_1967}. Thus, we obtain a triangulation of the manifold $M$.

    The simplicial maps $\phi_\Delta\circ \pi$ and $\operatorname{const}_{e_{n+1}}|_W$ glue to a simplicial map $f_\Delta$. It is easy to see that the maps $f$ and $f_\Delta$ are homotopic.
    

    We will construct the desired family of simplicial maps using the Pontrjagin construction. By Proposition~\ref{j-real}, any admissible value of the Hopf invariant is realized by a class in $\operatorname{Im} J \subset \pi_{2n-1}(S^n)$, so we may assume that $P = S^{n-1} \subset M = S^{2n-1}$ is a standardly embedded sphere. We can triangulate $P$ as the boundary of a simplex $\partial\Delta^n$.

    Let $s$ be an arbitrary maximal simplex of $S^n \cong \partial \Delta^{n+1}$. Then exactly one $n$-simplex $\tau$ maps into $s$ under $\phi_\Delta$. It is clear that $N(f,s)$ is the number of maximal simplices in $P \times \tau$.

    \[\#\operatorname{Max}(P \times \tau) = \#\operatorname{Max}(\partial\Delta^n)\cdot \#\operatorname{Max}(\Delta^{n-1} \times \tau)=(n+1)\binom{2n-1}{n}. \]
    
    Thus, the following theorem holds:
    \begin{theorem} \label{upper}
        For any admissible value $k$ (for a fixed $n$), there exists a simplicial complex $K$ triangulating the $(2n-1)$-sphere $|K| \cong S^{2n-1}$ and a simplicial map $f \colon K \rightarrow \partial \Delta^{n+1}$ such that $H(f) = k$, and 
        \[N(f,s) = (n+1)\binom{2n-1}{n}\]
        for any maximal face $s$ of the boundary simplex $\partial \Delta^{n+1}$.
    \end{theorem}

\section*{Acknowledgments} I would like to thank my advisor Sem{\"e}n Podkorytov for helpful discussions. I would also like to thank Gregory Arone, Fedor Petrov, Andrey Ryabichev and Fedor Vylegzhanin for their valuable remarks.

  \bibliographystyle{plain}
   \bibliography{reference}

\end{document}